\documentclass[11pt, reqno]{amsart}
\usepackage{amsmath, amsthm, amscd, amsfonts, amssymb, graphicx, color, enumerate}
\usepackage[bookmarksnumbered, colorlinks, plainpages]{hyperref}
\hypersetup{colorlinks=true,linkcolor=red, anchorcolor=green, citecolor=cyan, urlcolor=red, filecolor=magenta, pdftoolbar=true}
\definecolor{TITLE}{rgb}{0.0,0.0,1.0}
\definecolor{AUTHOR1}{rgb}{0.00,0.59,0.00}
\definecolor{AUTHOR2}{rgb}{0.50,0.00,1.00}
\definecolor{SECTION}{rgb}{0.50,0.00,1.00}
\definecolor{FOOTTITLE}{rgb}{0.00,0.50,0.75}
\definecolor{THM}{rgb}{0.7,0.3,0.3}
\definecolor{SEC}{rgb}{0.6,0.1,.5}

\makeatletter
\@namedef{subjclassname@2010}{%
  \textup{2010} Mathematics Subject Classification}
\makeatother
\newtheorem{theorem}{{\color{THM} Theorem}}[section]

\newtheorem{lemma}[theorem]{{\color{THM}Lemma}}
\newtheorem{proposition}[theorem]{{\color{THM}Proposition}}
\newtheorem{corollary}[theorem]{{\color{THM}Corollary}}
\theoremstyle{definition}
\newtheorem{definition}[theorem]{{\color{THM}Definition\ }}
\newtheorem{example}[theorem]{{\color{THM}Example}}

\newtheorem{remark}[theorem]{{\color{THM}Remark}}

\newcommand{\A}{\mathfrak A}
\newcommand{\B}{\mathcal B}
\newcommand{\D}{\mathcal D}
\newcommand{\M}{\mathfrak X}
\newcommand{\V}{\mathcal V}
\newcommand{\K}{\mathcal L}
\newcommand{\W}{\mathcal W}
\newcommand{\X}{\mathcal M}
\newcommand{\NN}{\mathcal N}

\newcommand{\N}{\mbox{\kh N}}
\newcommand{\bea}{\begin{eqnarray*}}
\newcommand{\eea}{\end{eqnarray*}}
\usepackage{amscd}

\begin{document}
\noindent \textcolor[rgb]{0.99,0.00,0.00}{}\\[.5in]
\title [Lie derivations on trivial extension algebras]{Lie derivations on trivial extension algebras}
\author[Mokhtari, Moafian, Ebrahimi Vishki]{A.H. Mokhtari$^1$, F. Moafian$^2$ and H.R. Ebrahimi Vishki$^3$}
\address{$^1, ^2$ Department of Pure Mathematics, Ferdowsi University
 of Mashhad, P.O. Box 1159, Mashhad 91775, Iran}
\email{{amirmkh2002@yahoo.com; fahimeh.moafian@yahoo.com}}
\address{$^3$ Department of Pure Mathematics and  Center of Excellence
in Analysis on Algebraic Structures (CEAAS), Ferdowsi University of
Mashhad, P.O. Box 1159, Mashhad 91775, IRAN.}
\email{vishki@um.ac.ir}
\begin{abstract} In this paper we provide some  conditions under which a Lie derivation on a trivial extension algebra is proper, that is, it can be expressed as  a sum of a derivation and a center valued map vanishing at commutators. We then apply our results for  triangular algebras. Some illuminating examples are also included.
\end{abstract}
\maketitle
\section{introduction}
Let $\A$ be a unital algebra (over a commutative unital ring ${\bf R}$) and $\M$ be an $\A-$bimodule. A linear  mapping $\D$ from  $\A$ into $\M$ is said to be a derivation if
\[\D(ab)=\D(a)b+a\D(b)\qquad (a, b\in \A).\]
A linear  mapping $\K:\A\rightarrow\M$ is called a Lie derivation if
\[\K[a,b]=[\K(a),b]+[a,\K(b)]\qquad (a, b\in \A),\]
where $[\cdot,\cdot]$ stands for the Lie bracket.
Trivially every  derivation is a Lie derivation. If $\D:\A\rightarrow\A$ is a derivation and
$\ell:\A\rightarrow Z(\A)(:={\rm the\ center\ of}\ \A)$ is a linear map, then $\D+\ell$ is a Lie derivation if and only if $\ell([a,b])=0,$ for all $a,b\in\A.$  Lie derivations of this form are called proper Lie derivations.
A problem that we are dealing with is studying those conditions on an algebra such that every Lie derivation on it is
proper. We say that an algebra $\A$ has Lie derivation property if every Lie derivation on  $\A$  is proper.

Martindale \cite{Ma} was the first one who showed that every Lie derivation on certain primitive ring is proper.
Cheung \cite{C1} initiated the study of various mappings on triangular algebras; in particular, he  investigated
the properness of Lie derivations on triangular algebras (see also \cite{C2, JQ, MoE}). Cheung's results \cite{C2} have recently extended by Du and Wang \cite{DW} for a generalized matrix algebras.  Wang \cite{W} studied Lie $n-$derivations on a unital algebra with a nontrivial idempotent. Lie triple derivations on a unital algebra with a nontrivial idempotent have recently investigated by Benkovi\v{c} \cite{B}.

In this paper we study Lie derivations on a trivial extension algebra. Let $\M$ be an $\A-$bimodule, then the direct product $\A\times\M$ together with the pairwise addition, scalar product and the algebra multiplication defined by
\[(a,x)(b,y)=(ab,ay+xb)\qquad (a, b\in\A,   x, y\in\M),\]
is a unital algebra  which is called a  trivial extension of $\A$ by $\M$ and will be denoted by $\A\ltimes\M$. For example, every triangular algebra ${\rm Tri}(\mathcal{A}, \M, \B)$ is a trivial extension algebra. Indeed, it can be identified with the trivial extension algebra $(\mathcal{A}\oplus\B)\ltimes\M;$ (see Sec. 3).

Trivial extension algebras are known as a rich source of (counter-)examples in various situations in
functional analysis. Some aspects of  (Banach) algebras of this type  have been investigated in \cite{BDL} and \cite{Z}. Derivations into  various duals of a trivial extension (Banach)
algebra studied in \cite{Z}. Jordan higher derivations on a trivial extension algebra are discussed in \cite{Mo} (see also \cite{G} and \cite{EV}).

The main aim of this paper is providing some conditions under which a trivial extension algebra has the Lie derivation property. We are mainly dealing with those  $\A\ltimes\M$ for which $\A$ enjoys a nontrivial idempotent $p$ satisfying
\begin{equation}\tag{$\bigstar$}
pxq=x,
\end{equation}
 for all $x\in\M,$ where $q=1-p.$ The triangular algebra  is the main example of a trivial extension algebra satisfying $(\bigstar)$.

 In Section 2, we characterize the properness of a Lie derivation on $\A\ltimes\M$ (Theorem \ref{0l}), from which we derive Theorem \ref{main1}, providing some sufficient conditions ensuring  the  Lie derivation property for $\A\ltimes\M$.
 In Section 3,  we apply our results for a triangular algebra, recovering the main results of   \cite{C2}.
\section{Proper Lie derivations on $\A \ltimes \M$}\label{2}
We commence  with the following elementary  lemma  describing  the structures of derivations and Lie derivations on a trivial extension algebra $\A\ltimes\M$.
\begin{lemma}\label{a}
Let $\A$ be a unital algebra and $\M$ be an $\A-$bimodule. Then every linear map
$\K:\A\ltimes \M\longrightarrow \A\ltimes \M$ has the presentation
\begin{eqnarray}\label{1}
\K(a,x)=(\K_\A(a)+T(x),\K_\M(a)+S(x))\qquad (a\in \A, x\in \M),
\end{eqnarray}
for some linear mappings  $\K_\A:\A\longrightarrow \A$, $\K_\M:\A\longrightarrow \M$, $T:\M\longrightarrow \A$ and
$S:\M\longrightarrow\M.$ Moreover,

$\bullet$ $\K$ is a Lie derivation if and only if
\begin{enumerate}[\hspace{1em}\rm (a)]
\item $\K_\A$ and $\K_\M$ are Lie derivations;
\item  $T([a, x])=[a, T(x)]$ {\rm and} $[T(x), y]=[T(y), x]$;
\item  $S([a,x])=[\K_{\A}(a),x]+[a,S(x)]$,
\end{enumerate}
for all $a\in\A, x, y\in\M.$

$\bullet$ $\K$ is a derivation if and only if
\begin{enumerate}[\hspace{1em}\rm (i)]
\item  $\K_\A$ and $\K_\M$ are  derivations;
\item $T(ax)=aT(x)$, $T(xa)=T(x)a$ {\rm and} $xT(y)+T(x)y=0$;
\item $S(ax)=aS(x)+\K_\A(a)x$ {\rm and} $S(xa)=S(x)a+x\K_\A(a)$,
\end{enumerate}
for all $a\in\A, x, y\in\M$.
\end{lemma}
It can be simply verified that the center $Z(\A\ltimes \M)$ of  $\A\ltimes \M$ is
\begin{eqnarray*}
Z(\A\ltimes \M)&=&\{(a,x); a\in Z(\A), [b,x]=0=[a,y]\ {\rm for\ all}\ b\in\A, y\in\M\}\\
&=&\pi_{\A}(Z(\A\ltimes \M))\times\pi_{\M}(Z(\A\ltimes \M)),
\end{eqnarray*}
where $\pi_{\A}:\A\ltimes \M\longrightarrow \A$ and $\pi_{\M}:\A\ltimes \M\longrightarrow \M$ are the natural
projections given by $\pi_{\A}(a,x)=a$ and $\pi_{\M}(a,x)=x$, respectively.

It should be noticed that, if $\A\ltimes \M$ satisfies $(\bigstar)$, then the equality $[p,x]=0$ implies  $x=0$, for any $x\in\M.$ This leads to   $\pi_{\M}(Z(\A\ltimes \M))=\{0\},$ and so
\begin{eqnarray}\label{center}
Z(\A\ltimes \M)&=&\{(a,0); a\in Z(\A), [a,x]=0\ {\rm for\ all}\ x\in \M \}\\
&=&\pi_{\A}(Z(\A\ltimes \M))\times\{0\}.\notag
\end{eqnarray}
Further, the property $(\bigstar)$ also implies the following simplifications on the module operations which will be frequently used in the sequel.
\begin{equation}\label{simple}
qx=0=xp,\ px=x=xq,\ ax=papx\ {\rm and}\  xa=xqaq\qquad (a\in\A, x\in\M).
\end{equation}

The following characterization theorem which is a generalization of  \cite[Theorem 6]{C2} studies  the properness of a Lie derivation on $\A\ltimes\M.$ Before proceeding, we recall that an $\A-$bimodule $\M$ is called $2-$torsion free if $2x=0$ implies $x=0$, for any $x\in\M.$
\begin{theorem}\label{0l}
Suppose that the trivial extension algebra $\A\ltimes \M$ satisfies $(\bigstar)$ and that  both $\A$ and $\M$ are $2-$torsion free. Then a Lie derivation $\K$  on $\A\ltimes \M$ of the form
\[\K(a,x)=(\K_\A(a)+T(x),\K_\M(a)+S(x))\qquad (a\in\A, x\in\M),\]
 is proper if and only if there exists a linear map $\ell_\A:\A\rightarrow Z(\A)$ satisfying the following conditions:
\begin{enumerate}[\hspace{1em}\rm (i)]
\item  $\K_\A-\ell_\A$ is a derivation on $\A$.
\item $[\ell_\A(pap), x]=0=[\ell_\A(qaq), x]$  for all $a\in\A, x\in\M.$
\end{enumerate}
\end{theorem}
\begin{proof}
By Lemma \ref{a} every Lie derivation on  $\A\ltimes \M$ can be expressed in the from
\[\K(a,x)=(\K_{\A}(a)+T(x),\K_{\M}(a)+S(x)),\]
where $\K_{\A}:\A\longrightarrow\A$, $\K_{\M}:\A\longrightarrow\M$ are Lie derivations and $T:\M\longrightarrow\A$, $S:\M\longrightarrow\M$ are linear mappings satisfying
\[T([a,x])=[a,T(x)],\quad  [T(x), y]=[T(y), x]\quad  {\rm and}\quad  S([a,x])=[\K_{\A}(a),x]+[a,S(x)],\]
 for  all  $a\in\A, x,y\in\M.$

To prove  ``if" part, we set
\[\D(a,x)=\big((\K_\A-\ell_\A)(a)+T(x),\K_\M(a)+S(x)\big)\quad {\rm and}\quad \ell(a,x)=(\ell_\A(a), 0)\ \ (a\in\A, x\in\M).\]
Then clearly  $\K=\D +\ell.$ That   $\ell$ is linear and $\ell(\A\ltimes \M)\subseteq Z(\A\ltimes \M)$ follows trivially from $\ell_\A(\A)\subseteq Z(\A)$ and \eqref{center}.  It remains to show that $\D$ is a derivation on $\A\ltimes\M.$  To do this we use  Lemma \ref{a}. It should be mentioned that in the rest of proof we frequently making use the equalities in \eqref{simple}. First we have,
\begin{eqnarray}\label{ax}
\nonumber S(ax)&=&S([pap,x])\\
\nonumber &=&[\K_\A(pap),x]+[pap,S(x)] \\
\nonumber &=&[(\K_\A-\ell_\A)(pap), x]+[\ell_\A(pap), x]+aS(x)\\
\nonumber &=&[\big((\K_\A-\ell_\A)(p)ap+p(\K_\A-\ell_\A)(a)p+pa(\K_\A-\ell_\A)(p)\big), x]\\
\nonumber &&+[\ell_\A(pap), x]+aS(x)\\
&=&(\K_\A-\ell_\A)(a)x+[\ell_\A(pap), x]+aS(x),
\end{eqnarray}
for all $a\in\A, x\in\M$. Now the condition $[\ell_\A(pap), x]=0$ implies that  $S(xa)=(\K_\A-\ell_\A)(a)x+aS(x)$. With a similar procedure as above, from  $[\ell_\A(qaq), x]=0$ we get $S(xa)=x(\K_\A-\ell_\A)(a)+S(x)a$ for all $a\in\A, x\in\M.$

 From the equality
 \begin{equation*}\label{01}
 T(x)=T([p,x])=[p,T(x)]=pT(x)-T(x)p\quad (x\in\M)
 \end{equation*}
   we arrive at  $yT(x)=0=T(x)y$ and so  $yT(x)+T(y)x=0$ for all $y, x\in\M.$ It also follows that $pT(x)p=0$, $qT(x)q=0$ and $qT(x)p=0$ for all $x\in\M;$ note that $\A$ is $2-$torsion free.

The equality
\begin{equation*}\label{02}
 0=T([qap,x])=[qap,T(x)]=qapT(x)-T(x)qap\quad (a\in\A, x\in\M)
 \end{equation*}
 gives $qapT(x)=T(x)qap$,  for all $a\in\A, x\in\M.$  The latter  relation together with the equality
\begin{equation*}\label{03}
 T(ax)=T[pa,x]=paT(x)-T(x)pa\quad (a\in\A, x\in\M)
 \end{equation*}
 lead us to  $T(ax)=pT(ax)q=paT(x)q=aT(x),$  for all $a\in\A, x\in\M.$ By a similar argument we get $T(xa)=T(x)a$  for all $a\in\A, x\in\M.$

Next, we  set $\phi(a)=\K_{\M}(paq)$, then $\phi$ is a derivation.  Indeed, for each $a, b\in\A,$
\begin{eqnarray*}
\phi(ab)&=&\K_{\M}(pabq)\\
&=&\K_{\M}([pa,pbq])+\K_{\M}([paq,bq])\\
&=&\K_{\M}(pa)pbq-pbq\K_{\M}(pa)+pa\K_{\M}(pbq)-\K_{\M}(pbq)pa\\
&&+\K_{\M}(paq)bq-bq\K_{\M}(paq)+paq\K_{\M}(bq)-\K_{\M}(bq)paq\\
&=&a\phi(b)+\phi(a)b.
\end{eqnarray*}
As $\M$ is $2-$torsion free, the identity
\[\K_{\M}(qap)=\K_{\M}([qap,p])=[\K_{\M}(qap),p]+[qap,\K_{\M}(p)]=-\K_{\M}(qap),\]
implies that $\K_{\M}(qap)=0$ for all $a\in\A$.

As $\K_{\M}([pap,qaq])=0$, for all $a\in\A$, we get,
\begin{equation}\label{k}
\K_{\M}(pap)qaq=-pap\K_{\M}(qaq).
\end{equation}
Substituting $a$ with $qaq+p$ (resp. $pap+q$) in \eqref{k}, leads to   $p\K_{\M}(qaq)q=-\K_{\M}(p)a$ (resp. $p\K_{\M}(pap)q=a\K_{\M}(p)$), for all $a\in\A$.  We use the latter relations to prove  that $\K_\M$ is   the sum of an inner derivation (implemented by $\K_{\M}(p)$) and $\phi$,  and so it  is a derivation. Indeed, for each $a\in\A,$
\begin{eqnarray*}
\K_{\M}(a)&=&\K_{\M}(pap)+\K_{\M}(qaq)+\K_{\M}(paq)\\
&=&p\K_{\M}(pap)q+p\K_{\M}(qaq)q+\phi(a)\\
&=&a\K_{\M}(p)-\K_{\M}(p)a+\phi(a).
\end{eqnarray*}
Now Lemma \ref{a} confirms that $\D$ is a derivation on $\A\ltimes\M,$ and so $\K$ is proper, as claimed.

For the converse, suppose that $\K$ is proper, that is,  $\K=\D+\ell$, where $\D$ is a derivation and $\ell$ is a center valued linear map on $\A\ltimes\M.$ Then, from \eqref{center}, we get  $\ell(\A\ltimes\M)\subseteq \pi_{\A}(Z(\A\ltimes \M))\times\{0\},$ and this  implies that   $\ell$ has the presentation  $\ell(a,x)=(\ell_\A(a),0)$ with $[\ell_A(a), x]=0,$ for all $a\in\A, x\in\M,$ for some linear map  $\ell_\A:\A\longrightarrow Z(\A).$   On the other hand,  $\K-\ell=\D$   is a derivation on $\A\ltimes\M$ and so,  by Lemma \ref{a}, $\K_\A-\ell_\A$ is a derivation on $\A,$ as required.
\end{proof}


Applying Theorem \ref{0l}, we come to the next main  result providing some sufficient conditions ensuring the Lie derivation property for $\A\ltimes\M$. Before proceeding, we introduce an auxiliary subalgebra  $\W_\A$ associated to an algebra $\A.$

For an algebra $\A$, we denote by $\mathcal{W}_\A$ the smallest subalgebra of $\A$ contains all commutators and idempotents. We are especially dealing with those algebras satisfying $\mathcal{W}_\A=\A.$ Some known examples of algebras satisfying $\mathcal{W}_\A=\A$ are:
 the full matrix algebra $\A=M_n(A), n\geq 2, $ where $A$ is a unital algebra, and also every simple unital algebra $\A$ with a nontrivial idempotent.
\begin{theorem}\label{main1}
Suppose that the trivial extension algebra $\A\ltimes \M$ satisfies $(\bigstar)$ and  that both $\A$ and $\M$ are $2-$torsion free. Then $\A\ltimes\M$ has Lie derivation property if the following two conditions are satisfied:
\begin{enumerate}[\hspace{1em}\rm (I)]
\item $\A$ has Lie derivation property.
\item The following two conditions  hold:
\begin{enumerate}[\hspace{1em}\rm (i)]
\item $\W_{p\A p}=p\A p$; or $Z(p\A p)=\pi_{p\A p}(Z(\A\ltimes\M)).$
\item $\W_{q\A q}=q\A q$; or  $Z(q\A q)=\pi_{q\A q}(Z(\A\ltimes\M)).$
\end{enumerate}
\end{enumerate}
\end{theorem}
\begin{proof} Let $\K$ be a Lie derivation on $\A\ltimes\M$ with the presentation
 as given in Lemma \ref{a}. Since $\K_\A$ is a Lie derivation  and  $\A$ has Lie derivation property,  there exists a linear map $\ell_\A:\A\rightarrow Z(\A)$ such that  $\K_\A-\ell_\A$ is a derivation on $\A$ (and so $\ell_\A$ vanishes on commutators of $\A$). It is enough to show that, under either conditions of  (II), $\ell_\A$ satisfies Theorem \ref{0l}(ii); that is,  $[\ell_\A(pap), x]=0=[\ell_\A(qaq), x]$  for all $a\in\A, x\in\M.$

 To prove the conclusion, we consider the subset  $\A'=\{pap: [\ell_\A(pap),x]=0,\ {\rm for\ all}\  x\in\M\}$ of $p\A p.$
 We are going to show that  $\A'$ is a subalgebra of $p\A p$ including  all idempotents and commutators of  $p\A p.$ First, we shall prove that $\A'$ is a subalgebra. That $\A'$ is an $\bf R-$sbmodule of $\A$ follows from the linearity of $\ell_\A.$ The following identity confirms that  $\A'$ is closed under multiplication.
 \begin{equation}\label{sub}
 [\ell_\A(papbp),x]=[\ell_\A(pap),bx]+[\ell_\A(pbp),ax]\qquad (a,b\in\A, x\in\M).
 \end{equation}
To prove \eqref{sub}, note that from the identity \eqref{ax} we have
\begin{equation}\label{sub1}
S(ax)= (\K_\A-\ell_\A)(a) x+[\ell_\A(pap), x]+aS(x)\quad (a\in\A, x\in\M).
\end{equation}
Applying \eqref{sub1} for $ab$ we have,
\begin{equation}\label{sub2}
S(abx)= (\K_\A-\ell_\A)(ab) x+[\ell_\A(pabp), x]+abS(x).
\end{equation}
On the other hand, since  $a[\ell_\A(pbp), x]=[\ell_\A(pbp), ax]$, we have,
\begin{eqnarray*}\label{sub3}
S(abx)&=& (\K_\A-\ell_\A)(a)bx+[\ell_\A(pap), bx]+aS(bx)\\
&=& (\K_\A-\ell_\A)(a)bx+[\ell_\A(pap), bx]+a (\K_\A-\ell_\A)(b) x+[\ell_\A(pbp), ax]+abS(x).\\
\end{eqnarray*}
  Using the fact that $\K_\A-\ell_\A$ is a derivation, then a comparison of   the latter equation and   \eqref{sub2} leads to
\[[\ell_\A(pabp), x]=[\ell_\A(pap), bx]+[\ell_\A(pbp), ax];\]
for all $a,b\in\A, x\in\M,$ which trivially implies \eqref{sub}.

Next, we claim that $\A'$ contains all idempotents of $p\A p.$ First note that, if one puts  $a=b$ in \eqref{sub}, then
 \begin{equation}\label{sub4}
 [\ell_\A((pap)^2),x]=[\ell_\A(pap),2ax]\qquad (a\in\A, x\in\M).
 \end{equation}
 This follows that
\begin{equation}\label{sub5}
 [\ell_\A((pap)^3),x]=[\ell_\A((pap)^2(pap)),x]=[\ell_\A(pap),3a^2x]\qquad (a\in\A, x\in\M).
 \end{equation}
 Suppose that $pap\in p\A p$ is an idempotent, that is, $(pap)^2=pap.$
 By \eqref{sub4} and \eqref{sub5}, we arrive at
 \begin{eqnarray*}
 [\ell_\A(pap),x]=[\ell_\A(3(pap)^2-2(pap)^3),x]&=&3[\ell_\A(pap),2ax]-2[\ell_\A(pap),3a^2x]\\
 &=&[\ell_\A(pap),\big(6(pap)-6(pap)^2\big)x]=0;
 \end{eqnarray*}
and this says that the idempotent $pap$ lies in $\A'.$

Further, that $\A'$ contains all commutatorts follows trivially from the fact that $\ell_\A$ vanishes on commutators.
We thus have proved that $\A'$ is a subalgebra of $p\A p$ contains all idempotents and commutators. Now the assumption
$\W_{p\A p}=p\A p$ in (i) gives  $\A'=p\A p$, that is, $[\ell_\A(pap),x]=0$ for every $a\in\A, x\in\M.$  A similar argument shows that, if   $\W_{q\A q}=q\A q$, then   $[\ell_\A(qaq),x]=0$ for every $a\in\A, x\in\M.$

Next, as $\ell_\A(p\A p)\subseteq Z(p\A p)$, the assumptions $Z(p\A p)=\pi_{p\A p}(Z(\A\ltimes\M))$ implies that    $\ell_\A(p\A p)\subseteq \pi_{p\A p}(Z(\A\ltimes\M))$. By \eqref{center}, the latter relation implies the reqirment $[\ell_\A(pap), x]=0$  for all $a\in\A, x\in\M.$ Similarly, the equality $Z(q\A q)=\pi_{q\A q}(Z(\A\ltimes\M))$ gives $[\ell_\A(qaq), x]=0$  for all $a\in\A, x\in\M;$ and this completes the proof.
\end{proof}
As the following example demonstrates, the Lie derivation property of $\A$ in Theorem \ref{main1} is essential.
\begin{example}\label{e1}
Let $\A$ be a unital algebra with a nontrivial idempotent $p$, which does not have   Lie derivation property.  Let  $\K_\A$ be a non-proper Lie derivation on $\A.$ Let $\M$ be an $\A-$bimodule such that $pxq=x$  and $[\K_\A(a),x]=0,$ for all $a\in\A, x\in\M$. Then a direct verification show that $\K(a,x)=(\K_{\mathfrak{A}}(a), 0), (a,x)\in\A\ltimes\M,$ defines a non-proper Lie derivation on $\A\ltimes\M.$ \\
To see a concrete example of a pair $\A$, $\M$ satisfying the aforementioned conditions, let $\A$ be the triangular matrix algebra as given in {\cite[Example 8]{C2}} and let $\M=\Bbb{R}$ equipped with the module operations $x\cdot (a_{ij})=xa_{11}$, $(a_{ij})\cdot x=a_{44}x$,  $(a_{ij})\in\mathfrak{A}$ and $x\in\Bbb{R}$.
\end{example}
The above example and Theorem \ref{main1} confirm  that, the  Lie derivation property of $\A$ plays a key role for the Lie derivation property of $\A\ltimes\M.$ In this respect, Lie derivation property of a unital algebra containing a nontrivial idempotent has already studied by   Benkovi\v{c} \cite[Theorem 5.3]{B} (see also  the case $n=2$ of  a result given  by Wang \cite[Theorem 2.1]{W}). About the Lie derivation property of a unital algebra with a nontrivial idempotent, we quote the following result from the first and third  authors \cite{ME}, which  extended the aforementioned results.
\begin{proposition}{\rm ({\cite[Corollary  4.3]{ME}}).}\label{idempotent}
Let $\A$ be a $2-$torsion free unital algebra with  a nontrivial idempotent $p$ and $q=1-p.$ Then $\A$ has Lie derivation property if the following three conditions hold:
\begin{enumerate}[\hspace{1em}\rm (I)]
\item {\small $Z(q\A q)=Z(\A)q$ and $p\A q$ is a faithful left $p\A p-$module; or $\mathcal{W}_{p\A p}=p\A p$ and $p\A q$ is a faithful left $p\A p-$module; or $p\A p$ has Lie derivation property and $\mathcal{W}_{p\A p}=p\A p.$}
\item {\small $Z(p\A p)=Z(\A)p$ and $q\A p$ is a faithful right $q\A q-$module; or $\mathcal{W}_{q\A q}=q\A q$ and $q\A p$ is a faithful right $q\A q-$module; or $q\A q$ has Lie derivation property and $\mathcal{W}_{q\A q}=q\A q.$}
\item One of the following assertions holds:
\begin{enumerate}[\hspace{1em}\rm (i)]
\item Either $p\A p$ or $q\A q$ does not contain nonzero central ideals.
\item $p\A p$ and $q\A q$ are domain.
\item Either $p\A q$ or $q\A p$ is strongly faithful.
\end{enumerate}
\end{enumerate}
\end{proposition}
It should also be remarked that   if $p\A q\A p=0$ and $q\A p\A q=0$, then the condition (III) in the above  proposition is superfluous and   can be dropped from the hypotheses, (see also  \cite[Remark 5.4]{B}).

One may  apply Proposition \ref{idempotent} to show that,  the algebra $\A=B(X)$ of bounded operators on a Banach space $X$ with dimension greater than 2, as well as, the full matrix algebra $\A=M_n(A)$, $n\geq 2$, where $A$ is a $2-$torsion free unital algebra, have the Lie derivation property.

Illustrating Theorem \ref{main1} and Proposition \ref{idempotent}, in the following we give an example of a trivial extension algebra having Lie derivation property (which is not a triangular algebra).
\begin{example}
We consider the next subalgebra $\A$  of $M_4(\Bbb{R})$ with the nontrivial idempotent $p$ as follows;
\[\A=\Bigg\{\left(\begin{array}{cccc}
a & 0 & 0 & 0\\
0 & b & u & 0\\
0 & 0 & c & 0\\
0 & 0 & 0 & d\\
\end{array}\right)| a,\, b,\, c,\, d,\, u\in\Bbb{R}\Bigg\},\qquad p=\left(\begin{array}{cccc}
0 & 0 & 0 & 0\\
0 & 0 & 0 & 0\\
0 & 0 & 1 & 0\\
0 & 0 & 0 & 0\\
\end{array}\right).\]
One can directly check that $p\A p\cong\mathbb R$ and $q\A q\cong\mathbb{R}^3$ (where the algebras $\mathbb R$ and ${\mathbb R}^3$ are equipped with their natural pointwise multiplications). In particular, $p\A p$, $q\A q$ have Lie derivation property, $\W_{p\A p}=p\A p$, $\W_{q\A q}=q\A q$ and $p\A p$ does not contain  nonzero central ideals.  Thus  $\A$ has Lie derivation property by virtue of Proposition \ref{idempotent}.\\
Further,  $\M=\Bbb{R}$ ia an  $\A-$bimodule furnished with the module operations as
 \[(a_{ij})\cdot x=a_{33}x,\quad  x\cdot (a_{ij})=xa_{22}\qquad ((a_{ij})\in\A, x\in\Bbb{R}).\] Then clearly the trivial extension algebra $\A\ltimes\Bbb{R}$ satisfies the condition $(\bigstar);$ that is, $pxq=x$ for all $x\in\Bbb{R}$.  So Theorem \ref{main1} guarantees that $\A\ltimes\Bbb{R}$ has Lie derivation property. It is worthwhile mentioning that  $\A\ltimes\Bbb{R}$ is not a triangular algebra. This can be directly verified  that, there is no  nontrivial idempotent $P\in\A\ltimes\Bbb{R}$ such that $P(\A\ltimes\Bbb{R})Q\neq 0$ and $Q(\A\ltimes\Bbb{R})P=0,$ where $Q=1-P$ (see \cite{C1}).
\end{example}
                                 \section{Application to triangular algebras}\label{4}
We recall that a triangular algebra ${\rm Tri}(\mathcal{A}, \M, \B)$ is an algebra of the form
\[{\rm Tri}(\mathcal{A}, \M, \B)=\left\{\left(\begin{array}{cc}
a & x \\
0 & b \\
\end{array}\right)|\hspace*{0.1cm} a\in\mathcal{A},\hspace*{0.1cm} x\in\M, \hspace*{0.1cm}b\in\B\right\},\]
whose algebra operations are just like $2\times 2-$matrix operations;  where $\mathcal{A}$ and $\B$ are unital algebras and $\M$ is an $(\mathcal{A}, \B)-$bimodule; that is, a left $\mathcal A-$module and a right $\B-$module. One can easily check that ${\rm Tri}(\mathcal{A}, \M, \B)$ is isomorphic to the trivial extension algebra $(\mathcal{A}\oplus\B)\ltimes \M$, where the algebra $\mathcal{A} \oplus \B$ has its usual pairwise operations and $\M$ as an $(\mathcal{A}\oplus \B)-$bimodule is equipped with the module operations
\[(a\oplus b)x=ax\quad  {\rm and}\quad x(a\oplus b)=xb
\qquad (a\in\mathcal{A}, b\in\B, x\in\M).\]
Furthermore,  the  triangular algebra
${\rm Tri}(\mathcal A, \M, \B)\cong(\mathcal{A}\oplus\B)\ltimes\M$ satisfies the condition $(\bigstar)$. Indeed,  $p=(1_{\mathcal A},0)$ is a nontrivial idempotent, $q=(0, 1_{\mathcal B})$ and  a direct verification shows that $pxq=x,$ for all $x\in\M.$ Further, in this case for $\A=\mathcal A\oplus\B$ we have,
\[p\A p\cong\mathcal A,\quad p\A q=0,\quad q\A p=0\ {\rm and}\ q\A q\cong\B.\]

It should be mentioned that in this case, for a Lie derivation $\K$  on $(\mathcal{A}\oplus\B)\ltimes\M$ with the presentation
 \[\K(a\oplus b, x)=(\K_{\mathcal{A}\oplus\B}(a\oplus b)+T(x), \ \K_\M(a\oplus b)+S(x))\quad ((a\oplus b)\in\mathcal A\oplus\B, x\in\M),\]
 as given in Lemma \ref{a}, we conclude that $T=0$. Indeed, by Lemma \ref{a}(b),  $T([a\oplus b, x])=[a\oplus b, T(x)]$ for all $a\in\mathcal A, b\in\B, x\in\M.$ Using the latter relation for $a=1, b=0$ implies that $T(x)=0$ for all $x\in\M.$

A quick look at the proof of Theorem \ref{0l} reveals that, in this special case, as $T=0$ and $q\A p=0$, we do not need the $2-$torsion freeness of $\A$ and $\M$ in Theorems  \ref{0l} and \ref{main1}.

A direct verification also reveals that, the direct sum $\A=\mathcal A\oplus\B$  ha Lie derivation property if and only if both $\mathcal A$ and $\B$ have Lie derivation property.

Now, by the above observations, as an immediate consequence  of Theorem \ref{main1}, we  directly arrive at the  following result of Cheung \cite{C2}.
\begin{corollary} [See {\cite[Theorem 11]{C2}}]\label{tri} Let $\mathcal{A}$ and $\B$ be  unital algebras and let $\M$ be  an $(\mathcal{A}, \B)-$bimodule.
Then the triangular algebra $\mathcal{T}={\rm Tri}(\mathcal{A}, \M, \B)$ has Lie derivation property if the following two conditions are satisfied:
\begin{enumerate}[\hspace{1em}\rm (I)]
\item $\mathcal A$ and $\B$ have Lie derivation property.
\item The following two conditions  hold:
\begin{enumerate}[\hspace{1em}\rm (i)]
\item $\W_{\mathcal A}=\mathcal A$; or $Z(\mathcal A)=\pi_{\mathcal A}(Z(\mathcal{T})).$
\item $\W_{\B}=\B$; or  $Z(\B)=\pi_{B}(Z(\mathcal{T})).$
\end{enumerate}
\end{enumerate}
\end{corollary}
It should be remarked  here that, in \cite{C2}, Cheung   combined his hypotheses  with some ``faithfulness'' conditions. Combining the conditions ``$\M$ is faithful as a left $\mathcal A-$module" and ``$\M$ is faithful as a right $\B-$module" with those in the above corollary provides some more sufficient conditions ensuring the Lie derivation property for the triangular algebra ${\rm Tri}(\mathcal{A}, \M, \B).$ His results can be satisfactorily extended to a trivial extension algebra $\A\ltimes\M$ by employing  the hypothesis  ``$\M$ is loyal" instead of ``$\M$ is faithful''.

We recall that, in the case where a unital algebra $\A$ has a nontrivial idempotent $p$,  an $\A-$bimodule $\M$ is said to be left loyal if
$a\M=0$ implies that $pap=0$, right loyal if $\M a=0$ implies that $qaq=0$, and it is called loyal if it is both left and right loyal.

Note that for a triangular algebra $(\mathcal A\oplus\B)\ltimes\M$, the  loyalty of $\M$ is nothing but the faithfulness of $\M$ as an $(\mathcal A,\B)$-module in the sense of Cheung \cite{C2}. Combining ``the loyalty of $\M$" with the current hypotheses of Theorem \ref{main1} provides some more sufficient conditions seeking the Lie derivation property for a trivial extension algebra  $\A\ltimes\M$. In the case where  $\M$ is a loyal $\A-$module,  the existence of an isomorphism between $\pi_{p\A p}(Z(\A\ltimes\M))$ and $\pi_{q\A q}(Z(\A\ltimes\M))$ is the  key tool. Indeed, it can be shown that,  there exists  a unique algebra isomorphism
$\tau:\pi_{p\A p}(Z(\A\ltimes\M))\longrightarrow \pi_{q\A q}(Z(\A\ltimes\M))$
 satisfying $papx=x\tau(pap)$ for all $a\in\A, x\in\M$; (see {\cite[Proposition 3]{C2}} in the setting of triangular algebra).


\begin{thebibliography}{99}
\bibitem{BDL}{\sc W.G. Bade, H.G. Dales and  Z.A. Lykova}, \textit{Algebraic and strong splittings of extensions of Banach algebras}, Mem. Amer. Math. Soc. \textbf{137},  (1999).

\bibitem{B}{\sc  D. Benkovi\v{c}}, \textit{Lie triple derivations of unital algebras with idempotents}, Linear Multilinear Algebra, \textbf{65} (2015), 141-165.

\bibitem{C1} {\sc W.-S. Cheung,} \textit{Mappings on triangular algebras}, PhD Dissertation, University of Victoria, (2000).

\bibitem{C2}{\sc W.-S. Cheung}, \textit{ Lie derivations of triangular algebras}, Linear   Multilinear Algebra  \textbf{51}  (2003), 299-310.

\bibitem{DW} {\sc Y. Du and Y. Wang,} \textit{ Lie derivations of generalized matrix algebras}, Linear Algebra Appl.
\textbf{437}  (2012), 2719-2726.

\bibitem{EV} {\sc A. Erfanian Attar and H.R. Ebrahimi Vishki,} \textit{Jordan derivations on trivial extension algebras}, preprint.

\bibitem{G}{\sc H. Ghahramani}, \textit{Jordan derivations on trivial extensions}, Bull. Iranian Math. Soc. \textbf{39}  (2013),  635-645.

\bibitem{JQ} {\sc P. Ji and W. Qi,} \textit{Charactrizations of Lie derivations of triangular algebras}, Linear Algebra Appl. \textbf{435} (2011), 1137-1146.


\bibitem{Ma}{\sc W.S. Martindale III,} \textit{Lie derivations of primitivr rings}, Michigan Math. J. \textbf{11}  (1964), 183-187.

\bibitem{Mo}{\sc F. Moafian}, \textit{Higher derivations on trivial extension algebras and triangular algebras,} PhD Thesis, Ferdowsi University of Mashhad, (2015).

\bibitem{MoE}{\sc F. Moafian and H.R. Ebrahimi Vishki}, \textit{Lie higher derivations on tiangular algebras revisited,} to appear in Filomat.

\bibitem{ME}{\sc A.H. Mokhtari and H.R. Ebrahimi Vishki}, \textit{More on Lie derivations of generalized matrix algebras,} preprint, see arXiv:1505.02344v1[math.RA].









\bibitem{W} {\sc Y. Wang,} \textit{ Lie $n-$derivations of unital algebras with idempotents},  Linear  Algebra Appl. \textbf{458}  (2014), 512-525.

\bibitem{Z}{\sc Y. Zhang}, \textit{Weak amenability of module extensions of Banach algebras},  Trans. Amer. Math. Soc. \textbf{354} (2002), 4131-4151.

\end{thebibliography}
\end{document}